\documentclass[a4paper,11pt]{amsart}
\usepackage[a4paper]{geometry}
\usepackage{amsmath,amsthm}
\usepackage{amssymb,amsfonts}
\usepackage{hyperref}
\usepackage{indentfirst}
\usepackage{tikz}
\usepackage{epigraph}
\usepackage{enumitem}   
\usepackage{graphicx}
\usepackage{calligra}
\usepackage{stmaryrd}
\usepackage{chngpage}
\usepackage{tikz-cd}
\usepackage{mathtools}
\DeclareSymbolFont{bbold}{U}{bbold}{m}{n}
\DeclareSymbolFontAlphabet{\mathbbm}{bbold}

\theoremstyle{plain}
	\newtheorem{theorem*}{Theorem}
	\newtheorem{theorem}{Theorem}[section]
\numberwithin{equation}{section}

	\newtheorem{proposition}[theorem]{Proposition}
	\newtheorem{lemma}[theorem]{Lemma}
	\newtheorem{corollary}[theorem]{Corollary}

\theoremstyle{definition}
	\newtheorem{definition}[theorem]{Definition}

	\newtheorem{question}[theorem]{Question}
	\newtheorem*{acknow}{Acknowledgements}
\theoremstyle{remark}

\numberwithin{equation}{section}

\newcommand{\cstar}{${C}^\ast$}

\newcommand{\bbR}{\mathbb{R}}
\newcommand{\bbN}{\mathbb{N}}
\newcommand{\e}{\varepsilon}
\newcommand{\cU}{\mathcal{U}}

\newcommand{\cN}{\mathcal{N}}

\newcommand{\cM}{\mathcal{M}}
\newcommand{\ctr}{\mathrm{ctr}}

\newcommand{\cLc}{\mathcal{L}_{C^*}}
\newcommand{\cLt}{\mathcal{L}_2}

\newcommand{\xMapsto}[2][]{\ext@arrow 0599{\Mapstofill@}{#1}{#2}}
\def\Mapstofill@{\arrowfill@{\Mapstochar\Relbar}\Relbar\Rightarrow}
\makeatother

\renewcommand{\phi}{\varphi}

\title{Ultraproducts of factorial $W^*$-bundles}

\author[Andrea Vaccaro]{Andrea Vaccaro}
\address{Mathematisches Institut, Fachbereich Mathematik und Informatik der
Universit\"at M\"unster, Einsteinstrasse 62, 48149 M\"unster, Germany.}
\email{avaccaro@uni-muenster.de}
\urladdr{https://sites.google.com/view/avaccaro}

\thanks{
The author was supported by the Deutsche Forschungsgemeinschaft (DFG, German Research Foundation) under Germany’s Excellence Strategy EXC 2044 –390685587, Mathematics M\"unster: Dynamics–Geometry–Structure, through SFB 1442, by the ERC Advanced Grant 834267 - AMAREC, and by the European Union’s Horizon 2020 research and innovation program under the Marie Sk\l odowska- Curie grant agreement No. 891709.}


\subjclass[2010]{}

\begin{document}
\maketitle
\begin{abstract}
This paper investigates {factorial} $W^*$-bundles and their ultraproducts. More precisely, a $W^*$-bundle
is \emph{factorial} if the von Neumann algebras associated to its fibers are all factors.
Let $\cM$ be the tracial
ultraproduct of a family of factorial $W^*$-bundles over compact Hausdorff spaces with finite, uniformly bounded
covering dimensions. We prove that in this case
the set of limit traces in $\cM$ is weak$^*$-dense in the trace space $T(\cM)$. This in particular entails that $\cM$ is {factorial}.
We also provide, on the other hand, an example of ultraproduct of factorial $W^*$-bundles which is not factorial.
Finally, we obtain some results of model-theoretic nature: if $A$ and $B$ are exact, $\mathcal{Z}$-stable \cstar-algebras,
or if they both have strict comparison, then
$A \equiv B$ implies that $T(A)$ is Bauer if and only if $T(B)$ is.  If moreover both $T(A)$ and $T(B)$ are Bauer simplices and second countable,
then the sets of extreme traces $\partial_e T(A)$ and $\partial_e T(B)$ have the same covering dimension.
\end{abstract}
\section{Introduction}
$W^*$-bundles, first introduced by Ozawa in \cite{ozawa_dix} as tracial $W^*$-analogues of \cstar-bundles
and of $C(X)$-algebras, are \cstar-algebras that arise as bundles over compact topological spaces and whose fibers are tracial von Neumann algebras (Definition \ref{def:wbundle}).

Due to their hybrid nature, $W^*$-bundles generally find their main \emph{raison d'être} in the role of bridge that they play between tracial von Neumann algebras
and stably finite \cstar-algebras. Indeed, Ozawa's foresighted intuition to isolate this class in \cite{ozawa_dix} was prompted by
the celebrated paper by Matui and Sato \cite{matuisato} and the subsequent series of work \cite{kirchror, sato, TWW}, investigating the Toms--Winter
Conjecture for \cstar-algebras whose trace space is a Bauer simplex. $W^*$-bundles arise as tracial completions (in the sense of \cite{ozawa_dix}, see Definition \ref{def:comp})
of such \cstar-algebras, and they have been used systematically for arguments relying on approximations and properties where tracial 2-norms appear,
with numerous applications also in the equivariant framework (\cite{BBSTWW, liao:Z, liao:Zm, wouters}).

In this note we investigate $W^*$-bundles from a more abstract point of view, as a class in its own right,
with an approach closer to that in \cite{ev_pen, sam_thesis}. The main motivation for the present paper is the forthcoming
work on tracially complete \cstar-algebras \cite{tracially_complete} by Carri\'on, Castillejos, Evington, Gabe, Schafhauser, Tikuisis and White.
We briefly pause to report some of the basic concepts and open problems
considered in their project in order to give the proper context and motivation to our results. We would like to thank the authors of \cite{tracially_complete} for allowing us
to include here some of the contents of their work, not yet publicly available at the time of writing this note.

The fundamental definition considered in \cite{tracially_complete} is that of \emph{tracially complete \cstar-algebra}, which provides an abstract
framework to study tracial completions of \cstar-algebras as defined in \cite{ozawa_dix}.

\begin{definition}[\cite{tracially_complete}] \label{def:tc}
Fix a \cstar-algebra $\cM$ and a non-empty set $X$ of $T(\cM)$, the set of tracial states of $\cM$. Consider the 2-seminorm
\[
\lVert a \rVert_{2, X} = \sup_{\tau \in X} \tau(a^*a)^{1/2}, \ a \in \cM.
\]
A \emph{tracially complete} \cstar-algebra is a pair $(\cM, X)$ where $\cM$ is a unital\footnote{The assumption of unitality is proved to be redundant in \cite{tracially_complete}.} \cstar-algebra and $X$ is a compact, non-empty, convex
subset of the trace space $T(\cM)$ of $\cM$ such that $\lVert \cdot \rVert_{2, X}$
is a norm on $\cM$ and such that the \cstar-norm unit ball of $\cM$ is $\lVert \cdot \rVert_{2,X}$-complete. A tracially complete \cstar-algebra $(\cM, X)$
is \emph{factorial} if moreover $X$ is a closed face of $T(\cM)$.
\end{definition}

The most elementary examples of tracially complete \cstar-algebras are tracial von Neumann algebras $(\cM, \tau)$,
where $\tau$ is a faithful normal trace. This is the scenario where $X = \{ \tau \}$. $W^*$-bundles form another important class
of tracially complete \cstar-algebras (see \S\ref{ss.wbundle}) and, in the factorial case, they correspond precisely to those $(\cM, X)$ for which $X$ is a Bauer simplex (this is a consequence of \cite[Theorem 3]{ozawa_dix}; see Theorem \ref{thm:bauer}).

A tracial von Neumann algebra $(\cM, \tau)$ is factorial as a tracially complete \cstar-algebra if and only if it is a factor, hence the name.
Part of the motivation why the class of factorial tracially complete \cstar-algebras has been isolated in \cite{tracially_complete}, is that such algebras tend to be more manageable and tractable than general tracially complete \cstar-algebras.
This emerges both in technical and elementary facts (such as Lemma \ref{lemma:traces}, Lemma \ref{lemma:fdlh_bundles} or Proposition \ref{prop:dixmier}),
as well as in more ambitious and sophisticated results like the classification theorems announced in \cite{tracially_complete}.

In this paper we address the following question, which appeared in an early version of
\cite{tracially_complete}.

\begin{question}[\cite{tracially_complete}] \label{ques:ultra}
Let $((\cM_i, X_i) \mid {i \in I})$ be a sequence of factorial tracially complete \cstar-algebras, let $\cM = \prod^\cU \cM_i$ be
the corresponding tracial ultraproduct, and let $X$ be the weak$^*$-closure of the set of all limit traces
on $\cM$. Is $(\cM, X)$ factorial?
\end{question}

Limit traces on $\cM$ are those that are obtained by taking $\cU$-limits of sequences of traces $(\tau_i)_{i \in I} \in \prod_{i \in I} X_i$ (see \S\ref{ss.ultra}).
The tracial ultraproduct $(\cM, X)$ considered in Question \ref{ques:ultra} is the \emph{right} notion of ultraproduct in the category
of tracially complete \cstar-algebras (see \S\ref{ss.ultra} for the case of $W^*$-bundles and \S\ref{ss.mt} for general tracially complete
\cstar-algebras), so Question \ref{ques:ultra} is simply asking whether factoriality
is preserved when passing to the ultraproduct. This is the case for tracial von Neumann algebras, indeed it is well-known that the
tracial ultraproduct of a family of finite factors (which in this case corresponds to the usual von Neumann ultraproduct) is again a finite factor
(see e.g. \cite{model_cstar:II}).

In \cite[Theorem 8]{ozawa_dix} Ozawa proved that, for ultraproducts of exact $\mathcal{Z}$-stable \cstar-algebras,
the set of limit traces is weak$^*$-dense in the trace space. Rephrased in the framework of Question \ref{ques:ultra},
what \cite[Theorem 8]{ozawa_dix} shows is that the ultraproduct $(\cM, X)$ of a sequence
of factorial tracially complete \cstar-algebras arising as tracial completions of exact $\mathcal{Z}$-stable \cstar-algebras, satisfies $X = T(\cM)$,
hence in particular it is factorial.
An analogous result is \cite[Proposition 2.5]{CETW}, form which it can be deduced that if $(\cM, X)$ is the tracial ultrapower
of a factorial tracially complete \cstar-algebra which is the tracial
completion of separable \cstar-algebra with \emph{complemented
partitions of unity}, then again $X = T(\cM)$, so $(\cM, X)$ is factorial.
The existence of complemented partitions of unity (usually referred to as \emph{CPoU}) is
a technical condition introduced in \cite[Definition 3.1]{CETWW}, which is automatic
for instance in $\mathcal{Z}$-stabile nuclear \cstar-algebras (\cite[Theorem 3.8]{CETWW}). We finally refer to \cite{APRT} for a recent
and more general account, employing Cuntz semigroup techniques, on when limit traces are weak$^*$-dense in the trace space of (\cstar-norm)
ultraproducts.

CPoU are an extremely
powerful tool in the study of tracially complete \cstar-algebras in \cite{tracially_complete}, effectively dividing these algebras in two subclasses,
a tamer one where the presence of CPoU allows to transfer numerous results and techniques from the theory of von Neumann algebras,
and its complement, much less understood. This paper focuses on the latter, while restricting to $W^*$-bundles.

The following theorem shows that Question \ref{ques:ultra} has affirmative answer for ultraproducts of $W^*$-bundles, even without complemented partitions of unity,
as long as their base spaces have bounded covering dimensions.

\begin{theorem} \label{thm:Wbundle_trace}
Let $( \cM_i )_{i \in I}$ be a sequence of factorial $W^*$-bundles over compact Hausdorff spaces
$K_i$. Suppose there is $d \in \bbN$ such that $\text{dim}(K_i) \le d$
for every $i \in I$. Let $\cM$ be the corresponding ultraproduct, which is a $W^*$-bundle over the ultracoproduct $\sum^\cU K_i$. Then the set of limit traces
is weak$^*$-dense in $T(\cM)$, and in particular $\cM$ is a factorial $W^*$-bundle.
\end{theorem}

We also prove that if the uniform bound on the covering dimension of $K_i$ is removed from Theorem \ref{thm:Wbundle_trace}, then its conclusion might fail.
In fact, we show that Question \ref{ques:ultra} has negative answer in general, even when restricted to $W^*$-bundles.
\begin{theorem} \label{thm:up}
There exists a sequence of factorial $W^*$-bundles whose ultraproduct is not factorial.
\end{theorem}

The sequence we use for Theorem \ref{thm:up} dates back to \cite{PP}, and it consists of 2-homogeneous \cstar-algebras arising
from certain vector bundles over finite-dimensional complex projective spaces. We remark that such family is the same as the one
considered in \cite{ozawa_dix}
to give an example of an ultraproduct for which the set of limit traces is not weak$^*$-dense in the whole trace space.


Note that the sequence considered in Theorem \ref{thm:up} is composed by $W^*$-bundles whose fibers are matrix
algebras, hence type I. This is in contrast with the primary focus of \cite{tracially_complete} and of most applications of $W^*$-bundles and tracial completions in the literature, which mainly concerns
tracially complete \cstar-algebras whose fibers are infinite-dimensional.
These are referred to as \emph{type II$_1$ tracially complete \cstar-algebras} in \cite{tracially_complete}, and it would be interesting
to know whether Question \ref{ques:ultra} has negative answer also for those algebras.

The final part of the paper has a model theoretic flavor, investigating how the first order theory of a \cstar-algebra can determine
the topological properties of its trace space. A precursor to the result below can be found in \cite[Section 3.5]{modelc},
where it is proved that, for classes of \cstar-algebras where the Cuntz--Pedersen nullset is definable (in the sense of
\cite[Chapter 3]{modelc}), being monotracial is preserved
by elementary equivalence.

\begin{theorem} \label{thm:model}
Let $A,B$ be two unital \cstar-algebras which are exact and $\mathcal{Z}$-stable, or which have strict comparison,
or which belong to any other class
where the Cuntz--Pedersen nullset is definable. Suppose that $A$ is elementarily equivalent to $B$. Then $T(A)$ is a Bauer simplex if and only if 
$T(B)$ is. If moreover both $T(A)$ and $T(B)$ are Bauer simplices and second countable,
then $\partial_e T(A)$ and $\partial_e T(B)$ have the same covering dimension.
\end{theorem}

The proof of Theorem \ref{thm:model} uses a tracial analogue of Dixmier's avergaing property for
factorial $W^*$-bundles (Proposition \ref{prop:dixmier}), which permits us to show that the center of the ultraproduct of a sequence of factorial
$W^*$-bundles is isomorphic to the ultraproduct of the centers (Theorem \ref{thm:center}).

The paper is structured as follows. In \S\ref{s.pre} we present definitions and preliminaries needed in later sections, \S\ref{s.fd} is devoted
to the proof of Theorem \ref{thm:Wbundle_trace}, while \S\ref{s.up} is where Theorem \ref{thm:up} is proved. Finally,
\S\ref{s.mt} contains the background and proofs needed to show Theorem \ref{thm:model}

\begin{acknow}
{I am grateful to the authors of \cite{tracially_complete} for sharing with me a preliminary version of their work and for allowing me
to include some parts of it in the present paper. I would moreover like to thank them for some helpful feedback on an early draft of this manuscript. Finally, I wish to thank the anonymous referee for carefully reading this note and for their useful remarks.}
\end{acknow}

\section{Preliminaries} \label{s.pre}

Given a \cstar-algebra $A$, denote by $A_1$, $A_{sa}$ and $A_+$ respectively the set of all contractions, self-adjoint and positive elements
in $A$. We let $Z(A)$ denote the center of $A$ and, given $a,b \in A$, we abbreviate the commutator $ab - ba$ as $[a,b]$.

The \emph{trace space} $T(A)$ of $A$ is the set of all tracial states on $A$, which we refer to simply as \emph{traces}.
For $\tau \in T(A)$ define the 2-seminorm
$\lVert \cdot \rVert_{2, \tau}$ on $A$ as
\[
\lVert a \rVert_{2,\tau} = \tau(a^*a)^{1/2}, \text{ for every } a \in A.
\]
Given a non-empty $X \subseteq T(A)$, the 2-seminorm on $A$ associated to $X$ is
\[
\lVert a \rVert_{2, X} = \sup_{\tau \in X} \lVert a \rVert_{2,\tau}, \text{ for every } a \in A.
\]
Note that $\lVert \cdot \rVert_{2, X} = \lVert \cdot \rVert_{2, \overline{\text{conv}}(X)}$, where $\overline{\text{conv}}(X)$ is the closed convex hull
of $X$. Given a convex set $X$, we let $\partial_e X$ denote the set of extreme points of $X$. The equality $\lVert \cdot \rVert_{2, X} = \lVert \cdot \rVert_{2, \partial_e X}$ follows by the Krein--Milman Theorem.

Given a unital \cstar-algebra $A$, the trace space $T(A)$, as well as any other of its closed, convex subsets $X$, is a \emph{Choquet simplex}
(\cite[Section 3]{alfsen}). In particular, for every $x \in X$ there exists a unique boundary measure $\mu_x$ (in the sense of \cite[Proposition I.4.5]{alfsen})
such that $f(x) = \int_X f(t) d\mu_x$ for every continuous affine function $f \colon X \to \bbR$.
A \emph{Bauer simplex} $X$ is a Choquet simplex such that $\partial_e X$ is closed.

\subsection{$W^*$-bundles} \label{ss.wbundle}
\begin{definition}[{\cite[Section 5]{ozawa_dix}}] \label{def:wbundle}
A \emph{$W^*$-bundle over (a compact Hausdorff space) $K$} is a unital \cstar-algebra $\cM$, with a unital embedding of $C(K)$
in the center of $\cM$ and a faithful unital conditional expectation $E \colon \cM \to C(K)$ such that
\begin{enumerate}
\item $E$ is tracial, that is $E(ab) = E(ba)$ for all $a,b \in \cM$,
\item \label{wbundle:i2} the \cstar-norm unit ball of $\cM$ is complete with respect of the norm $\lVert \cdot \rVert_{2,K}$ induced by $E$, defined as $\lVert a \rVert_{2, K} = \lVert
E(a^*a) \rVert^{1/2}$, for all $a \in \cM$.
\end{enumerate}
\end{definition}

Let $\cM$ be a $W^*$-bundle over $K$ with conditional expectation $E \colon \cM \to C(K)$, and fix $\lambda
\in K$. Throughout the paper we shall denote by $\tau_\lambda$ the trace $\text{ev}_\lambda \circ E$ and by $\pi_\lambda$ the
GNS-representation corresponding to $\tau_\lambda$. The von Neumann algebra $\pi_\lambda(\cM)''$ is the \emph{fiber} corresponding
to $\lambda$.

More generally, every regular Borel probability measure $\mu$ over $K$ naturally
induces a trace $\tau_\mu$ defined for $a \in \cM$ as
\[
\tau_\mu(a) = \int_K E(a) \, d\mu.
\]
Let $X = \{ \tau_\mu \mid \mu \in \text{Prob}(K) \}$. Faithfulness of $E$ entails that $\lVert \cdot \rVert_{2,X}$ is a norm on $\cM$,
while item \ref{wbundle:i2} of Definition \ref{def:wbundle} implies that $(\cM, X)$
is a tracially complete \cstar-algebra, since $\lVert \cdot \rVert_{2,X} = \lVert \cdot \rVert_{2,\partial_e X} = \lVert \cdot \rVert_{2,K}$.
We sometimes identify $X$ and $\partial_e X$ with $ \text{Prob}(K)$ and $K$ respectively. Note that the notations
$\lVert \cdot \rVert_{2,\partial_e X}$ and $\lVert \cdot \rVert_{2,K}$ are consistent with this identification.
We always implicitly consider the $W^*$-bundle $\cM$ as a tracially complete \cstar-algebra, in particular we say that $\cM$ is \emph{factorial} if the pair
$(\cM, \text{Prob}(K))$ is a factorial tracially complete \cstar-algebra.

The following proposition isolates some useful reformulations of factoriality. Its statement, as well as its proof, originates from some analogous statements appearing in an
early version of \cite{tracially_complete}.
\begin{proposition} \label{prop:factor}
Let $(\cM, X)$ be a tracially complete \cstar-algebra. The following conditions are equivalent.
\begin{enumerate}
\item $(\cM, X)$ is factorial.
\item $\partial_e X \subseteq \partial_e T(\cM)$.
\item $\pi_\lambda(\cM)''$ is a factor for every $\lambda \in \partial_e X$.
\end{enumerate}
\end{proposition}
\begin{proof}
(1) $\Rightarrow$ (2) is true since every extreme point of a face must be extreme in the simplex itself.

For (2) $\Rightarrow$ (1), let $F = \text{conv}(\partial_e X)$ be the convex hull of $\partial_e X$. Given $x \in F$ then $x \in \text{conv}(x_1, \dots, x_n)$ for some $x_1, \dots, x_n \in \partial_e X$. The
set $\text{conv}(x_1, \dots, x_n)$ is the (closed) convex hull of a compact subset of $\partial_e X \subseteq \partial_e T(\cM)$, hence it is a face
of $T(\cM)$ by \cite[Corollary 11.1.19]{goodearl}. It follows that if $x = \lambda y + (1-\lambda)z$ for $\lambda \in (0,1)$ and
$y,z \in T(\cM)$, then $y,z \in \text{conv}(x_1, \dots, x_n) \subseteq F$, hence $F$ is a face of $T(\cM)$. Finally, $X = \overline{F}$ is face by \cite[Proposition 4.4]{roy}.

(2) $\Leftrightarrow$ (3) follows by the well-known fact that,
for $\tau \in T(\cM)$ and $\pi_\tau$ the corresponding GNS-representation, $\pi_\tau(\cM)''$ is a factor if and only if $\tau \in \partial_e T(\cM)$.
\end{proof}

Summarizing, a $W^*$-bundle $\cM$ over $K$ gives rise to a tracially complete \cstar-algebra $(\cM, X)$ where $X$ is a Bauer simplex whose boundary is homeomorphic to $K$.
The converse also holds in the factorial case: every factorial tracially complete \cstar-algebra whose base space is a Bauer simplex can be naturally endowed
with a $W^*$-bundle structure. This fact is direct consequence of \cite[Theorem 3]{ozawa_dix}, and it is stated and proved below in a form which 
is due to \cite{tracially_complete}.
\begin{theorem}[{\cite[Theorem 3]{ozawa_dix}, \cite{tracially_complete}}] \label{thm:bauer}
Let $(\cM, X)$ be a factorial tracially complete \cstar-algebra such that $X$ is a Bauer simplex.
Then there exists an embedding $\theta \colon C(\partial_e X) \to Z(\cM)$ such that 
\[
\tau(\theta(f)a) = \int_{\partial_e X} f(\sigma) \sigma(a) \, d\mu_\tau(\sigma), \text{ for every } \tau \in X, f \in C(\partial_e X), a \in \cM.
\]
Moreover $(\cM, X)$ can be endowed with the structure of a $W^*$-bundle over $\partial_e X$ with the conditional
expectation $E \colon \cM \to C(\partial_e X)$ defined as $E(a)(\tau) = \tau(a)$ for $\tau \in \partial_e X$.
\end{theorem}
\begin{proof}
The existence of an embedding of $\theta\colon C(\partial_e X) \to Z(\cM)$ as claimed in the statement is a consequence of \cite[Theorem 3]{ozawa_dix}. The latter result is proved for $X$
metrizable face (which appears as $S$ in the notation of \cite{ozawa_dix}), where metrizability is required only to make sure that $\partial_e X$ is Borel. In our context this is automatic, since $\partial_e X$ is assumed to be closed. Given $a \in \cM$, let $\hat a \in C(\partial_e X)$ be defined as $\hat a(\tau) = \tau(a)$ for every $\tau \in \partial_e X$.
Let $E$ be the conditional expectation $E(a) = \theta(\hat a)$.
It is immediate to check that the 2-norm induced by $E$ is the same as $\lVert \cdot \rVert_{2,\partial_e X} = \lVert \cdot \rVert_{2, X}$.
By assumption $\cM_1$ is $\lVert \cdot \rVert_{2, X}$-complete, making $\cM$ a $W^*$-bundle over $\partial_e X$.
\end{proof}

\subsection{Ultrapowers and ultraproducts of $W^*$-bundles} \label{ss.ultra}
Fix an infinite index set $I$ and a free ultrafilter $\cU$ over $I$. Let $( (A_i, X_i)  \mid {i \in I})$ be a sequence of pairs where
$A_i$ is a unital \cstar-algebra and $X_i$ is a non-empty subset of $T(A_i)$ for every $i \in I$. The \emph{tracial ultraproduct} of such
sequence is the \cstar-algebra
\[
\prod\nolimits^\cU (A_i, X_i) = \frac{\prod\nolimits_{i \in I} A_i }{\{ (a_i)_{i \in I} \in  \prod\nolimits_{i \in I} A_i  \mid \lim_{i \to \cU}
\lVert a_i \rVert_{2, X_i} = 0 \}}.
\]
In case of a constant sequence $(A_i, X_i) = (A, X)$, we use the notation $A^\cU_X$ and refer to this \cstar-algebra as the
\emph{tracial ultrapower} of $(A,X)$. We write $A^\cU$ if $X = T(A)$.
Throughout this paper we shall only consider cases where $\lVert \cdot \rVert_{2, X_i}$ is a norm for every $i \in I$.

We drop $X_i$ from the notation when it is clear from the context and simply write $\prod^\cU A_i$. For instance,
let $( \cM_i )_{i \in I}$ be $W^*$-bundles over $K_i$. Then by $\prod^\cU \cM_i$ we mean
the tracial ultraproduct of the sequence $((\cM_i, K_i) \mid {i \in I})$.

Given $( (A_i, X_i)  \mid {i \in I})$,
every sequence of traces $\bar \tau= (\tau_i)_{i \in I} \in \prod_{i \in I} X_i$ determines a trace on $\prod^\cU A_i$ defined on
each representing sequence as
\[
\bar \tau((a_i)_{i \in I}) = \lim_{i \to \cU} \tau_i(a_i).
\]
We denote by $\prod_\cU X_i$ the set of traces which arise in this manner, and we refer to them as \emph{limit traces}. $\prod_\cU X_i$
corresponds to the set-theoretic ultraproduct of $( X_i )_{i \in I}$. This is a convex, not necessarily closed, subset of the trace space of $\prod^\cU A_i$. Let $\sum^\cU X_i$ denote its weak$^*$-closure. In case $X_i = X$ for every $i \in I$, we abbreviate such closure
as $X^\cU$. When every $X_i$ is compact (e.g. in the case of $W^*$-bundles),
the space $\sum^\cU X_i$ can also be obtained as the \emph{ultracoproduct} of the sequence $(X_i )_{i \in I}$, namely
the compact Hausdorff space such that $C(\sum^\cU X_i) \cong \prod_\cU C(X_i)$, where the latter (with $\cU$ in subscript) is the canonical
\cstar-norm ultraproduct.

The unit ball of the tracial ultraproduct $\prod^\cU A_i$ is complete with respect to the 2-norm $\lVert \cdot \rVert_{\sum^\cU X_i}$
(see \S\ref{ss.mt}), in particular for ultraproducts
of $W^*$-bundles we have the following.

\begin{proposition}[{\cite[Proposition 3.9]{BBSTWW}}] \label{prop:ultra_w}
Let $(\cM_i )_{i \in I}$ be a sequence of $W^*$-bundles over $K_i$ with conditional expectations $E_i \colon \cM_i \to C(K_i)$.
The tracial ultraproduct $\prod^\cU \cM_i$ is a $W^*$-bundle over $\sum^\cU K_i$, with the conditional expectation
\begin{align*}
E^\cU \colon \prod\nolimits^\cU \cM_i &\to C\Big (\sum\nolimits^\cU K_i\Big ) \\
(a_i)_{i \in I} &\mapsto (E_i(a_i))_{i \in I}
\end{align*}
inducing the norm $\lVert \cdot \rVert_{\sum^\cU K_i}$.
\end{proposition}

\section{$W^*$-Bundles with Finite-Dimensional Base Space} \label{s.fd}
Despite employing different techniques, both \cite[Theorem 8]{ozawa_dix} and \cite[Proposition 2.5]{CETW} 
are proved by showing first that elements which
are small with respect of all traces can be approximated with sums of commutators, with the number of summands not changing,
or at least being kept under control, as the precision of the approximation varies.
Our Theorem \ref{thm:Wbundle_trace} is no exception, and is based on the following lemma.

\begin{lemma} \label{lemma:fdlh_bundles}
Let $\cM$ be a factorial $W^*$-bundle over the space $K$.
Suppose that $K$ has finite covering dimension $d$, and let $a \in \cM_{sa}$ be a contraction such that $E(a) = 0$.
Then for ever $\e > 0$ there exist contractions $w_i, z_i \in \cM$ for $i = 1, \dots, 10d$ such that
\[
\Big \lVert a -  24 \sum_{i \le 10d} [w_i, z_i] \Big\rVert_{2,K} < \e.
\]
\end{lemma}
\begin{proof}
Let
$\lambda \in K$. Recall that $\tau_\lambda$ denotes the trace $\text{ev}_\lambda \circ E$ on $\cM$ and that $\pi_\lambda$ denotes the GNS-representation corresponding to $\tau_\lambda$. Fix $a \in \cM_{sa}$ as in the statement. As $\cM$ is factorial, by Proposition \ref{prop:factor}
the von Neumann algebra $\pi_\lambda(\cM)''$ is a factor. The assumption $E(a)(\lambda) = 0$ thus entails that $\pi_\lambda(a)$
is mapped to zero by the unique trace on $\pi_\lambda(\cM)''$.
By \cite[Theorem 2.3]{fackdlh} there exist contractions $\tilde w^\lambda_1, \dots, \tilde w^\lambda_{10},
 \tilde z^\lambda_1, \dots, \tilde z^\lambda_{10} \in \pi_\lambda(\cM)''$ such that
\[
\pi_\lambda(a) = 24 \sum_{k \le 10} [\tilde w^\lambda_k, \tilde z^\lambda_k].
\]

By the Kaplansky Density Theorem, for every $k  \le 10$ there are contractions $ w^\lambda_k,  z^\lambda_k \in \cM$ approximating $\tilde w^\lambda_k$ and $\tilde z^\lambda_k$ well enough so that
\[
\Big \lVert a - 24 \sum_{k \le 10} [w^\lambda_k, z^\lambda_k]  \Big \rVert_{2, \tau_\lambda} = \Big \lVert \pi_\lambda(a) -  24\sum_{k \le 10} [\pi_\lambda(w^\lambda_k), \pi_\lambda(z^\lambda_k)]  \Big \rVert_{2, \tau_\lambda}  < \e.
\]

By continuity of the 2-norm, for every $\lambda \in K$ there exists an open neighborhood $U_\lambda$ of $\lambda$ such that 
\begin{equation} \label{eq:local}
\Big \lVert a - 24\sum_{k \le 10} [w^\lambda_k, z^\lambda_k]  \Big \rVert_{2, \tau_{\lambda'}} < \e, \text{ for every } \lambda' \in U_\lambda.
\end{equation}
By compactness of $K$ there exists a finite open cover $\mathcal{V}$ of $K$ where for each $U \in \mathcal{V}$
there is $\lambda_U \in K$ such that $U = U_{\lambda_U}$. Moreover, as $K$ has covering dimension equal to $d$,
$\mathcal{V}$ can be partitioned as $\mathcal{V}_0 \sqcup \dots \sqcup \mathcal{V}_d$ so that the elements of each $\mathcal{V}_j$ are
pairwise disjoint (\cite[Lemma 3.2]{blan_kirch}).

Let $\{ f_U \}_{U \in \mathcal{V}} \subseteq C(K) \subseteq \cM$ be a partition of the unity on $K$ such that $\text{supp}(f_U)
\subseteq U$ for every $U \in \mathcal{V}$. For every $j = 0,\dots , d$ and $k = 1,\dots, 10$ define the following elements of $\cM$
\[
w_{k + 10j}= \sum_{U \in \mathcal{V}_j} f^{1/2}_U w_k^{\lambda_U},
\]
\[
z_{k + 10j}= \sum_{U \in \mathcal{V}_j} f^{1/2}_U z_k^{\lambda_U}.
\]
Note that, as the functions $\{f_U\}_{U \in \mathcal{V}_j}$ have pairwise disjoint support, the elements defined above are still contractions and the verify the equality
\begin{equation} \label{eq:orth}
[w_{k +10j}, z_{k+10j}] = \sum_{U \in \mathcal{V}_j} f_U [w_k^{\lambda_U}, z_k^{\lambda_U}],
\end{equation}
for every $j = 0,\dots , d$ and $k = 1,\dots, 10$.

We claim that $\{w_i, z_i \}_{i \le 10d}$ are the desired elements. Indeed, for any $\lambda \in K$ and $ 0 \le j \le d$ there is at most
one $U_j \in \mathcal{V}_j$ such that $\lambda \in U_j$ (if there is none, simply pick a random $U_j \in \mathcal{V}_j$), hence 
\begin{align*}
\Big\lVert a - 24\sum_{i \le 10d} [w_i, z_i] \Big\rVert_{2, \tau_\lambda} &\stackrel{\mathclap{\eqref{eq:orth}}}{=} \Big\lVert \sum_{j} \sum_{U \in \mathcal{V}_j} f_U(\lambda)a - 24\sum_{k,j}\sum_{U \in \mathcal{V}_j} f_U(\lambda)[w^{\lambda_U}_k, z^{\lambda_U}_k] \Big\rVert_{2, \tau_\lambda} \\
& = \Big\lVert \sum_{j}  f_{U_j}(\lambda)a -24 \sum_{k,j} f_{U_j}(\lambda)[w^{\lambda_{U_j}}_k, z^{\lambda_{U_j}}_k] \Big\rVert_{2, \tau_\lambda} \\
& \le \sum_{j}  f_{U_j}(\lambda) \Big\lVert a - 24\sum_k [w^{\lambda_{U_j}}_k, z^{\lambda_{U_j}}_k]  \Big\rVert_{2, \tau_\lambda} \\
&\stackrel{\mathclap{\eqref{eq:local}}}{<} \sum_{j}  f_{U_j}(\lambda) \e \\
&\le \e.
\end{align*}
We conclude that
\[
\Big\lVert a - 24\sum_{i \le 10d} [w_i, z_i] \Big\rVert_{2, K} = \sup_{\lambda \in K}  \Big\lVert a - 24\sum_{i \le 10d} [w_i, z_i] \Big\rVert_{2, \tau_\lambda} < \e.
\]
\end{proof}

We first prove Theorem \ref{thm:Wbundle_trace} for the case $I = \bbN$, to make it more accessible for readers not well acquainted
with model theory.
An elementary model-theoretic argument (which we defer to \S\ref{ss.mt}) shows that its conclusion holds for ultraproducts
over arbitrary sets of indices.
\begin{theorem} \label{thm:Wbundle_trace_real}
Let $( \cM_n )_{n \in \bbN}$ be a sequence of factorial $W^*$-bundles over compact Hausdorff spaces
$K_n$. Suppose there is $d \in \bbN$ such that $\text{dim}(K_n) \le d$
for every $n \in \bbN$. Let $\cM = \prod^\cU \cM_n$ be the corresponding ultraproduct. Then the set of limit traces
is weak$^*$-dense in $T(\cM)$, and in particular $\cM$ is a factorial $W^*$-bundle.\end{theorem}
\begin{proof}
By Proposition \ref{prop:ultra_w} the ultraproduct $\cM = \prod^\cU \cM_n$ is a $W^*$-bundle over $K = \sum^\cU K_n$.
Arguing as in \cite[Lemma 4.4]{CETWW} and \cite[Theorem 8]{ozawa_dix}, by an application of the Hahn--Banach Theorem
it is sufficient to show  that the following equality holds for every $a \in \cM_{sa}$ 
\begin{equation} \label{eq:equality}
\sup_{\lambda \in K} | \tau_\lambda(a) | = \sup_{\tau \in T(\cM)} | \tau(a) |.
\end{equation}
Fix thus a contraction $a \in \cM_{sa}$ and suppose then that $\sup_{\lambda \in K} | \tau_\lambda(a) | \le \delta$ for some $\delta \ge 0$.
The equality in \eqref{eq:equality} follows if we can provide $c, w_1, \dots, w_{10d}, z_1, \dots, z_{10d} \in \cM$ such that $\lVert c \rVert \le \delta$ and
\begin{equation} \label{eq:commutants}
a - c = \sum_{i \le 10d} [w_i, z_i].
\end{equation}

Let $(a_n)_{n \in \bbN}$ be a representative sequence of self-adjoint contractions for $a$. Up to a rescaling of $a_n$ we can assume that
$\sup_{\lambda \in K_n} | \tau_\lambda(a_n) | \le \delta$ for all $n \in \bbN$. This gives 
\begin{equation} \label{eq:norm}
\lVert E_n(a_n) \rVert \le \delta, \text{ for every } n \in\bbN.
\end{equation}
Moreover by Lemma \ref{lemma:fdlh_bundles} there exist $w_{1,n}, \dots, w_{10d,n}, z_{1,n} \dots, z_{10d,n} \in \cM_n$
of norm bounded by 48 (since $\lVert a_n - E_n(a_n) \rVert \le 2$) such that
\[
\Big\lVert a_n - E_n(a_n) - \sum_{i \le 10d} [w_{i,n}, z_{i,n}] \Big\rVert_{2, K_n} < \frac{1}{n}, \text{ for every } n \in \bbN.
\]

The elements $c = (E_n(a_n))_{n \in \bbN}$, $w_i = (w_{i,n})_{n \in \bbN}$ and $z_i = (z_{i,n})_{n \in \bbN}$ satisfy the equality
in \eqref{eq:commutants}, and $\lVert c \rVert \le \delta$ by \eqref{eq:norm} as desired.
\end{proof}

\section{A non-factorial Ultraproduct} \label{s.up}

The example we provide for Theorem \ref{thm:up} is an ultraproduct of homogeneous \cstar-algebras. The fact
that unital homogeneous \cstar-algebras
can be endowed with a structure of $W^*$-bundle over their spectrum is proved in Proposition \ref{prop:homo}, and it is a direct consequence of the Dauns--Hoffman Theorem (\cite[Theorem A.34]{RaeWill}).

A \cstar-algebra is \emph{$n$-homogeneous} if every irreducible representation has dimension $n$.
A \cstar-algebra is \emph{homogeneous} if it is $n$-homogeneous for some $n \in \bbN$.
Homogeneous \cstar-algebras have continuous trace (\cite[Proposition IV.1.4.14]{blackadar}) and, in the unital case, their spectrum, namely the
set of all irreducible
representations up to unitary equivalence, is compact and Hausdorff with the Jacobson topology.
For a \cstar-algebra $A$, let $\hat A$ denote its spectrum. When $A$ has continuous-trace, the spectrum $\hat A$ is homeomorphic
to the \emph{primitive ideal space} of $A$.
In what follows we shall thus identify every $t \in \hat A$ with the corresponding primitive ideal on $A$.
Given $a \in A$ we denote by $a(t)$ the class
of $a$ in the quotient $A /t$.

\begin{proposition} \label{prop:homo}
Let $A$ be a unital $n$-homogeneous \cstar-algebra. Then $A$ is a factorial $W^*$-bundle over $\hat A$
with conditional expectation $E \colon A \to C(\hat A)$ defined as $E(a)(t) = \text{tr}_n(a(t))$, where $\text{tr}_n$ is the normalized trace on
$n \times n$ matrices.
\end{proposition}
\begin{proof}
Since $A$ has continuous trace and $A/t \cong M_n$ for every $t \in \hat A$,
the function $\hat a$ mapping $t \in \hat A$ to $\text{tr}_n(a(t))$ is continuous over $\hat A$ for every
$a \in A$. By the Dauns--Hoffman Theorem \cite[Theorem A.34]{RaeWill} there exists an isomorphism $\theta$
of $C(\hat A)$ onto the center $Z(A)$ of $A$ such that
\begin{equation} \label{eq:DH}
(\theta(f)a)(t) = f(t)a(t), \text{ for every } t \in \hat A, \, f \in C(\hat A), \, a \in A.
\end{equation}
As a consequence of these facts, the map
\begin{align*}
E \colon A &\to C(\hat A) \\
a &\mapsto \hat a
\end{align*}
is a tracial conditional expectation of $A$ onto $Z(A)$ (up to $\theta$). The map $E$ is moreover faithful since $\text{tr}_n$ is faithful
and for every $a \in A_+$ there exists some $t \in \hat A$ such that $a(t) > 0$. 

We claim next that every extremal trace on $A$ is of the form $\text{ev}_t \circ E$ for some $t \in \hat A$. To see this, fix $\tau \in \partial_e T(A)$, and let $\pi_\tau$ be the
corresponding GNS-representation. The center $Z(A)$
is mapped by $\pi_\tau$ into $Z(\pi_\tau(A)'')$ which, as $\tau$ is extremal, only consists of scalars. The restriction of $\pi_\tau$
to $\theta(C(\hat A))$ is hence a point evaluation, meaning that there is $s \in \hat A$ such that $\pi_\tau(\theta(f))=0$ whenever $f \in C(\hat A)$
verifies $f(s) = 0$. This fact can be used to show that $\pi_\tau$ factors through the quotient map $A \to A/s$. Indeed, given
$a \in A$ such that $a(s) =0$ and $\epsilon > 0$,
we can find an open neighborhood $U$ of $s$ in $\hat A$ such that $\lVert a(t) \rVert < \e$ for every $t \in U$ (\cite[Lemma 5.2.b]{RaeWill}).
Let next $g \in C(\hat A)$ of norm 1 be such that $g(s) = 0$ and $g\restriction{\hat A \setminus U} \equiv 1$. It follows that
\[
\lVert a - \theta(g) a \rVert \stackrel{\eqref{eq:DH}}{=} \sup_{t \in \hat A} \lVert a(t) - g(t) a(t) \rVert < 2\e.
\]
Since $g(s) = 0$ it follows that $\theta(g)a \in \ker \pi_\tau$. We have thus showed that $a$ can be approximated with elements in
$\ker \pi_\tau$, so $\pi_\tau(a) = 0$, and both $\pi_\tau$ and $\tau$ factor through $A/s$. Since the latter admits a unique
trace, we conclude that $\tau(a) = E(a)(s)$ for every $a \in A$.

In order to conclude that $A$ is a $W^*$-bundle, we need to prove that the unit ball of $A$
is complete with respect to the 2-norm induced by $E$. Note that since every extremal trace on $A$ is captured
by $E$, then the 2-norm induced by $E$ is equal to $\lVert \cdot \rVert_{2, T(A)}$.
As $A$ is $n$-homogeneous, the \cstar-norm $\lVert \cdot \rVert$ and $\lVert \cdot \rVert_{2, T(A)}$ are equivalent, in fact we have
\begin{equation} \label{eq:equivalent}
n^{-1/2} \lVert a \rVert \le \lVert a \rVert_{2, T(A)} \le \lVert a \rVert, \text{ for every } a \in A.
\end{equation}
The inequality $\lVert a \rVert_{2,T(A)} \le \lVert a \rVert$ is always verified. For the other inequality, as
 $A / t \cong M_n$ for every $t \in \hat A$, we have that
\[
\lVert a(t) \rVert \le n^{1/2} \text{tr}_n(a^*a(t))^{1/2} =n^{1/2} \widehat{a^*a}(t)^{1/2} , \text{ for every } a \in A.
\]
We thus conclude that, for every $a \in A$
\[
\lVert a \rVert = \sup_{t \in \hat A} \lVert a(t) \rVert \le n^{1/2} \lVert E(a^*a) \rVert^{1/2} = n^{1/2} \lVert a \rVert_{2,T(A)}.
\]
This implies that $A$ is $\lVert \cdot \rVert_{2, T(A)}$-complete.

Finally, we verify the third condition of Proposition \ref{prop:factor} to prove that $A$ is factorial.
Every trace $\tau$ of the form $\text{ev}_t \circ E$, for $t \in \hat A$ annihilates on the primitive
ideal $t$. $A/t$ is simple, hence $\pi_\tau(A) \cong A/t \cong M_n$, which in turn implies that $\pi_\tau(A)'' \cong M_n$ is a factor, and thus that the trace $\tau$ is extremal.
\end{proof}

Using the arguments in the previous proof, one can furthermore deduce that the map from $\hat A$ to $\partial_e T(A)$ sending $t \mapsto \text{ev}_t \circ E$ is a
homeomorphism.

The sequence of factorial $W^*$-bundles that we consider to prove Theorem \ref{thm:up} goes
back to \cite{PP}. We briefly recall their definition here, and we refer to the discussion preceeding \cite[Leema 3.5]{PP} and to \cite[Section 2]{BiFa}
for all the missing details. Given $n \in \bbN$, let $A_n$ be the \cstar-algebra of continuous sections of the following vector bundle
over the $n$-dimensional complex projective space $\mathbb{C}P^n$
\[
B_n = \left\{ \left (x, \begin{pmatrix}
a & \textbf{b} \\ \textbf{c} & d \end{pmatrix} \right) : a,d \in \mathbb{C}, \textbf{b} \in x, \bar{\textbf{c}} \in x \right\},
\]
where $\bar{\textbf{c}}= \overline{(c_1, \dots, c_{n+1})} = (\bar c_1, \dots, \bar c_{n+1})$, with multiplication and adjoint defined pointwise as
\[
\begin{pmatrix}
a & \textbf{b} \\ \textbf{c} & d \end{pmatrix} 
\begin{pmatrix}
a' & \textbf{b}' \\ \textbf{c}' & d' \end{pmatrix}
=
\begin{pmatrix}
aa' +  \textbf{b} \cdot \textbf{c}' & a\textbf{b}' + d' \textbf{b} \\ a'\textbf{c} + d \textbf{c}' & dd' + \textbf{b}'\cdot \textbf{c} \end{pmatrix} 
\text{ and }
\begin{pmatrix}
a & \textbf{b} \\ \textbf{c} & d \end{pmatrix}^*
=
\begin{pmatrix}
\bar a & \bar{\textbf{c}} \\ \bar{\textbf{b}} & \bar d \end{pmatrix}.
\]
All fibers of this bundle are isomorphic to $M_2$ (see e.g. \cite[p. 201]{PP}), hence each $A_n$ is a 2-homogeneous \cstar-algebra.

We recall that the \cstar-norm ultraproduct of $\{A_n \}_{n \in \bbN}$ is defined as
\[
\prod\nolimits_\cU A_n =\frac{\prod_n A_n}{\{ (a_n)_{n \in \bbN} \in \prod_n A_n \mid \lim_{n \to \cU} \lVert a_n \rVert = 0\}}.
\]
As pointed out by Ozawa before stating \cite[Theorem 8]{ozawa_dix}, the \cstar-norm ultraproduct $\cM = \prod_\cU A_n$ is an example where
the set of limit traces is not weak$^*$-dense in $T(\cM)$.
The reason for this is the existence, for every $n \in \bbN$, of non-zero projections
$p_n, q_n \in A_n$ such that $p_n - q_n$ can be approximated by a finite sum of commutators of the form $a^*a - aa^*$,
but it cannot be approximated by
sums of less than $n+1$ such commutators (see \cite[Lemma 3.5]{PP} or \cite[Theorem 2.1]{BiFa}).
As a consequence, the projections $p=(p_n)_{n \in \bbN}$ and $q = (q_n)_{n \in \bbN}$ in $\cM$ are such that $p - q$ is evaluated as zero
on every limit trace, but, on the other hand, it cannot be approximated by a finite sum of commutators in $\cM$. That is saying that
$p -q$ does not belong to the \emph{Cuntz--Pedersen nullset} $\cM_0$ of $\cM$ (\cite{cun_ped}, see also \S\ref{ss.mt}), hence by
\cite[Proposition 2.7]{cun_ped} the weak$^*$-closure of the set of limit traces does not exhaust $T(\cM)$.

This setup, combined with some elementary arguments, provides an answer to Question \ref{ques:ultra},
even in the setting of $W^*$-bundles.
\begin{corollary}
There exists a sequence of factorial $W^*$-bundles whose ultraproduct is not factorial.
\end{corollary}
\begin{proof}
Let $(A_n)_{n \in \bbN}$ be the sequence of \cstar-algebras discussed above. Since $A_n$ is unital and 2-homogeneous for every $n \in \bbN$,
by Proposition \ref{prop:homo} each $A_n$ can be naturally endowed with a structure of factorial $W^*$-bundle over $\partial_e T(A_n)
\cong \hat A_n \cong \mathbb{C}P^n$. The tracial ultraproduct $\prod^\cU A_n$ is a $W^*$-bundle over the space $K = \sum^\cU \hat A_n$
by Proposition \ref{prop:ultra_w}. Every $A_n$ is 2-homogeneous, hence the quotient map from \cstar-norm ultraproduct
$\cM = \prod_\cU A_n$ onto the tracial ultraproduct $\prod^\cU A_n$ is an isomorphism. Indeed
the kernel of the quotient map is
\[
\Big \{ (a_n)_{n \in \bbN} \in \prod\nolimits_\cU A_n \mid \lim_{n \to \cU} \lVert a_n \rVert_{2, T(A_n)} = 0 \Big\},
\]
which in this case is equal to the set of those $(a_n)_{n \in \bbN}$ such that $\lim_{n \to \cU} \lVert a_n \rVert = 0$,
since for every $n$, arguing as in the proof of Proposition \ref{prop:homo}, we have
\[
\lVert a \rVert_{2, T(A_n)} \le \lVert a \rVert \le 2^{1/2}\lVert a \rVert_{2, T(A_n)}, \text{ for every } a \in A_n.
\]
This also entails that the \cstar-norm and $\lVert \cdot \rVert_{2, K} = \lim_{n \to \cU} \lVert \cdot \rVert_{2, T(A_n)}$ on $\cM$ are equivalent, in particular
\begin{equation} \label{eq:2norm}
\lVert a \rVert_{2, K} \le \lVert a \rVert \le 2^{1/2}\lVert a \rVert_{2, K}, \text{ for every } a \in \cM.
\end{equation}

Identify $K$ with a subspace of $T(\cM)$, and let  $X$ be its closed convex hull, namely the weak$^*$-closure of the set of all limit traces on $\cM$ (the unconvinced reader may take a look at the discussion leading to equation \eqref{eq:wultra}). To show that $(\cM, X)$ is not factorial,
we argue by contradiction and suppose that $X$ is a closed face of $T(\cM)$. Let $\sigma \in T(\cM) \setminus X$,
whose existence is guaranteed by the discussion preceding the statement of the corollary. By \cite[Corollary II.5.20]{alfsen} and \cite[Lemma 6.2]{kirchror} there exists $a \in \cM_+$ such that
\begin{enumerate}
\item $\tau(a) < 1/4$ for every $\tau \in X$,
\item \label{i2:asa} $\sigma(a) > 1/2$.
\end{enumerate}
Then $\lVert a^{1/2} \rVert^2_{2, X} <  1/4$, thus by \eqref{eq:2norm} it follows that
$\lVert a^{1/2} \rVert^2 < 1/2$, which is a contradiction since, by item \ref{i2:asa} above,
\[
\lVert a^{1/2} \rVert^2 = \lVert a \rVert \ge \sigma(a) > 1/2.
\]
\end{proof}

\section{Center of the Ultraproduct and Consequences in Model Theory} \label{s.mt}
We start this section with a preliminary result (Theorem
\ref{thm:center}), a generalization of \cite[Corollary 4.3]{model_cstar:I}, where we show that the center of the ultraproduct of a family of
factorial $W^*$-bundles is the ultraproduct of the centers. We then proceed to prove Theorem \ref{thm:model}.

\subsection{Center of the Ultraproduct} \label{ss.center}
Theorem \ref{thm:center} is a consequence Proposition \ref{prop:dixmier}, stating
that factorial $W^*$-bundles verify a tracial analogue of the \emph{strong Dixmier's averaging property} (\cite[Definition III.2.5.16]{blackadar}).
This fact is not explicitly stated nor proved in \cite{ozawa_dix}, but it follows from arguments analogous to those appearing in that paper and was known to the author, who uses it to show \cite[Theorem 15]{ozawa_dix}. We give a full proof for the reader's convenience, starting with
some preliminary lemmas. The argument for the first one is due to the authors of \cite{tracially_complete}.

\begin{lemma}[{\cite{tracially_complete}}] \label{lemma:traces}
Let $A$ be a unital \cstar-algebra and let $X \subseteq T(A)$ be non-empty. Set $\pi = \bigoplus_{\tau \in X} \pi_\tau$
and let $\sigma \in T(\pi(A)'')$. Then the trace $\sigma \circ \pi$ on $A$ belongs to the closed face generated by $X$ in $T(A)$.
\end{lemma}
\begin{proof}
Let $F$ be the closed face generated by $X$ in $T(A)$. It is sufficient to prove the lemma for normal traces
since $F$ is closed and the set of normal traces is weak$^*$-dense in $T(\pi(A)'')$. Let then  $\sigma \in T(\pi(A)'')$ be normal
and suppose that $\sigma \circ \pi \notin F$. By \cite[Corollary II.5.20]{alfsen} there exists a continuous affine function
$f \colon T(A) \to [0,1]$ such that $f(\sigma \circ \pi) = 1$ and $f\restriction F \equiv 0$. By \cite[Lemma 6.2]{kirchror} there is
a sequence $(a_n)_{n \in \bbN} \subseteq A_+$ such that 
\[
\lim_{n \to \infty} \sup_{\tau \in T(A)} \lvert \tau(a_n) - f(\tau) \rvert = 0.
\]
This means that $\tau(a_n) \to 0$ for every $\tau \in F$ which, as $a_n \ge 0$ for every $n \in \bbN$, implies that $\pi(a_n)$ converges
to zero in the strong topology. The trace $\sigma$ is normal, hence $\sigma(\pi(a_n))$ also converges to zero,
which contradicts the fact
\[
\lim_{n \to \infty} \sigma(\pi(a_n)) = f(\sigma \circ \pi) = 1.
\]
\end{proof}

\begin{lemma} \label{lemma:ctr}
Let $A$ be a unital \cstar-algebra and let $X \subseteq T(A)$ be a non-empty closed face. Let $\pi = \bigoplus_{\tau \in X} \pi_\tau$
and let $\ctr \colon \pi(A)'' \to Z(\pi(A)'')$ be the center-valued trace on $\pi(A)''$. Then
\[
\lVert a \rVert_{2, X} = \lVert \ctr(\pi(a^*a)) \rVert^{1/2}, \ \forall a \in A.
\]
\end{lemma}
\begin{proof}
By \cite[Theorem III.2.5.7]{blackadar} the map from the state space of $Z(\pi(A)'')$ to $T(\pi(A)'')$ sending
$\phi$ to $\phi \circ \ctr$ is a bijection, hence
\[
 \lVert \ctr(b^*b) \rVert^{1/2} = \lVert b \rVert_{2, T(\pi(A)'')}, \text{ for every } b \in \pi(A)''.
\]
The conclusion of the lemma with $b = \pi(a)$ for $a \in A$ follows by Lemma \ref{lemma:traces}.
\end{proof}

\begin{lemma} \label{lemma:ctrE}
Let $\cM$ be a factorial $W^*$-bundle over $K$ with conditional expectation $E \colon \cM \to C(K)$. Let $\pi = \bigoplus_{\mu \in {\text{Prob}}(K)} \pi_{\tau_\mu}$, let $\cN = \pi(\cM)''$ and denote by $\text{ctr}$ the center-valued trace of $\cN$. Then, for every $a \in \cM$
\[
\pi(E(a)) = \text{ctr}(\pi(a)).
\]
\end{lemma}
\begin{proof}
Suppose that $\pi(E(a)) \not = \text{ctr}(\pi(a))$ for some $b \in\cM$. By \cite[Theorem III.2.5.7]{blackadar}  there exists $\tau \in T(\cN)$ such that
$\tau(\pi(E(a))) \not = \tau(\text{ctr}(\pi(a)))$. As $\cM$ is factorial, by Lemma \ref{lemma:traces} there is
$\mu \in \text{Prob}(K)$ such that $\tau \circ \pi = \tau_\mu$. In particular $\tau_\mu \circ E = \tau_\mu$, hence on the one hand we have
\[
\tau(\pi(E(a))) = \tau_\mu(E(a)) = \tau_\mu(a).
\]
On the other hand
\[
\tau(\text{ctr}(\pi(a))) = \tau(\pi(a)) = \tau_\mu(a),
\]
which is a contradiction.
\end{proof}

\begin{proposition} \label{prop:dixmier}
Let $\cM$ be a factorial $W^*$-bundle over $K$ with conditional expectation $E \colon \cM \to C(K)$. For every $a \in \cM$ and $\e > 0$ there are unitaries $u_1, \dots, u_k \in \cM$ such that
\[
\Big \lVert E(a) - \frac{1}{k} \sum_{i = 1}^k u_i a u_i^* \Big \rVert_{2, K} < \e.
\]
\end{proposition}
\begin{proof}
The proof follows closely the one of \cite[Theorem 3]{ozawa_dix}.
Fix $a \in \cM$ and $\e > 0$. Let $\pi$ be
the direct sum $\bigoplus_{\mu \in {\text{Prob}(K)}} \pi_{\tau_\mu}$ and set
$\cN = \pi (\cM)''$. Denote by $\ctr \colon \cN \to Z(\cN)$ the center-valued trace of $\cN$.
By the Dixmier Averaging Theorem (\cite[Theorem III.2.5.19]{blackadar})
and Lemma \ref{lemma:ctrE}, the element $\pi(E(a))$ belongs to the norm closure of the convex hull of
 $\{ u \pi(a) u^* \mid u \in \cU(\cN)\}$.
 
 Let $C = \{ \pi(u a u^*) \mid u \in \cU(\cM) \}$.
 By the Kaplansky Density Theorem every $u \in \cU(\cN)$ is strong limit of a net
of unitaries from $\pi(\cM)$. Since the adjoint operation is
strongly continuous on normal operators (\cite[Theorem 4.3.1]{murphy}) and multiplication
is strongly continuous on bounded sets (\cite[Remark 4.3.1]{murphy}), there exists a net $\{b_\lambda \}_{\lambda}$, where every $b_\lambda$ belongs to
convex hull of $C$, which converges to $b= \pi(E(a))$ in the strong topology. This in turn entails that the net
$\{(b_\lambda - b)^*(b_\lambda - b)\}_\lambda$ converges to 0 in the weak operator topology and thus, being
a bounded net, in the ultraweak topology. Since the center-valued trace $\ctr$ is ultraweakly continuous, this in particular
implies
\begin{equation} \label{eq:normal}
\phi\Big(\ctr((b_\lambda - b)^*(b_\lambda - b))\Big) \to 0, \text{ for every } \phi \in Z(\cN)_*,
\end{equation}
where $Z(\cN)_*$ denotes the set of normal functionals on $Z(\cN)$.

We claim that \eqref{eq:normal} implies
\begin{equation} \label{eq:nonormal}
\phi\Big(\ctr((b_\lambda - b)^*(b_\lambda - b))\Big) \to 0, \text{ for every } \phi \in Z(\cN)^*.
\end{equation}
Indeed, let $\phi \in Z(\cN)^*$. Then $\phi \circ \ctr \in T(\cN)$, hence by Lemma \ref{lemma:traces} there is  $\mu \in \text{prob}(K)$
such that $\phi \circ \ctr \restriction \pi(\cM) = \tau_\mu$. The latter extends to a normal trace on $\cN$, hence there is $\phi' \in Z(\cN)_*$
such that $\phi \circ \ctr \restriction \pi(\cM) = \phi' \circ \ctr \restriction \pi(\cM)$. We conclude that $\phi$ and $\phi'$ are equal
on $\{ \ctr((b_\lambda - b)^*(b_\lambda - b)) \}_\lambda$ by Lemma \ref{lemma:ctrE}, and therefore \eqref{eq:nonormal} follows.

By the Hahn--Banach Theorem there are thus finitely many $\alpha_j > 0$ with $\sum_j \alpha_j = 1$ such that
\begin{equation} \label{eq:small}
\Big \lVert \sum_j \alpha_j \text{ctr}((b_{\lambda_j} - b)(b_{\lambda_j} - b)^*) \Big \rVert < \e.
\end{equation}
Set $c = \sum_j \alpha_j b_{\lambda_j}$ and note that
\[
c=
\begin{bmatrix}
\alpha^{1/2}_1 \dots \ \alpha^{1/2}_m
\end{bmatrix}
\begin{bmatrix}
\alpha^{1/2}_1 b_{\lambda_1} \\
\vdots \\
\alpha^{1/2}_m b_{\lambda_m}
\end{bmatrix} =:rs.
\]
Hence $c^*c \le s^*r^*rs \le \| r \|^2 s^*s = \sum_j \alpha_jb_{\lambda_j}^* b_{\lambda_j}$, which in turn gives
\begin{align*}
\ctr((c - b)^*(c - b)) & =\ctr (c^*c - b^*c -c^*b + b^*b) \\
& \le \ctr\Big(\sum_j \alpha_jb_{\lambda_j}^* b_{\lambda_j} -b^* \sum_j \alpha_j b_{\lambda_j} -  \sum_j \alpha_j b_{\lambda_j}^*b  +b^*b\Big) \\
&= \ctr\Big(\sum_j \alpha_j (b_{\lambda_j} - b)^*(b_{\lambda_j} - b)\Big).
\end{align*}

As a consequence, by \eqref{eq:small}, $\lVert \ctr((c - b)^*(c - b)) \rVert < \e$.
Summarizing, there are unitaries $u_1, \dots, u_\ell \in \cM$ and $\beta_i\ge 0$ with $\sum_{i \le \ell} \beta_i = 1$
such that $c = \pi(\sum_{i\le \ell} \beta_i u_i a u_i^*)$, and by Lemma \ref{lemma:ctr} and the inequalities above it follows that
\[
\Big \lVert E(a) - \sum_{i \le \ell} \beta_{i} u_i a u_i^* \Big\rVert^2_{2,K} = \lVert \text{ctr}((b-c)^*(b-c)) \rVert < \e.
\]
\end{proof}

We need one last lemma before showing Theorem \ref{thm:center} (see also \cite[Lemma 4.2]{model_cstar:I}).
\begin{lemma} \label{lemma:bound}
Let $\cM$ be a factorial $W^*$-bundle over $K$ with conditional expectation $E \colon \cM \to C(K)$.
Then for every $a \in \cM$ the following holds
\[
\lVert a - E(a) \rVert_{2,K} \le \sup_{b \in \cM_1} \lVert [a,b] \rVert_{2,K} \le 2 \lVert a - E(a) \rVert_{2,K}.
\]
\end{lemma}
\begin{proof}
Let $a \in \cM$ and $b \in \cM_1$. The right-hand side inequality in the statement follows from the computation below
\begin{align*}
\lVert ab- ba \rVert_{2,K} &\le \lVert ab - E(a)b \rVert_{2,K} + \lVert bE(a) - ba \rVert_{2,K} \\
&\le \lVert a- E(a) \rVert_{2,K} \lVert b \rVert + \lVert b \rVert \lVert E(a) - a \rVert_{2,K}\\
&\le 2 \lVert E(a) - a \rVert_{2,K}.
\end{align*}

For the other inequality, given $\epsilon > 0$, by Proposition \ref{prop:dixmier} there exist unitaries $u_1, \dots, u_k \in \cM$ such that
\begin{equation} \label{eq:dix_approx}
\Big \lVert E(a) - \frac{1}{k} \sum_{i = 1}^k u_i a u_i^*\Big \rVert_{2,K} < \e.
\end{equation}
We thus have
\begin{align*}
\lVert a - E(a) \rVert_{2,K} & < \Big\lVert a - \frac{1}{k} \sum_{i = 1}^k u_i a u_i^*\Big \rVert_{2,K} + \e \\
& \le \frac{1}{k} \sum_{i = 1}^k \lVert au_i - u_ia \rVert_{2,K} + \e \\
& \le \sup_{b \in \cM_1} \lVert [a,b] \rVert_{2,K} + \e.
\end{align*}
\end{proof}

\begin{theorem} \label{thm:center}
Let $(\cM_i)_{i \in I}$ be a sequence of factorial $W^*$-bundles. Then
\[
Z\Big(\prod\nolimits^\cU \cM_i\Big) = \prod\nolimits^\cU Z(\cM_i).
\]
\end{theorem}
\begin{proof}
The inclusion $\prod\nolimits^\cU Z(\cM_i) \subseteq Z(\prod\nolimits^\cU \cM_i)$ always holds.
For the reverse inclusion, let $(a_i)_{i \in I} \in Z(\prod\nolimits^\cU \cM_i)$. It suffices to show that $(a_i)_{i \in I} = (E_i(a_i))_{i \in I}$. For every $i \in I$ let $b_i \in (\cM_i)_1$ be such that $2 \lVert [a_i,b_i] \rVert_{2,K_i}
\ge \sup_{c \in (\cM_i)_1} \lVert [a_i,c] \rVert_{2,K_i}$. Lemma \ref{lemma:bound} gives
\[
\lVert a_i - E_i(a_i) \rVert_{2,K_i} \le \sup_{c \in (\cM_i)_1} \lVert [a_i,c] \rVert_{2,K_i} \le 2 \lVert [a_i, b_i] \rVert_{2,K_i}, \text{ for every } i \in I.
\]
The right-most term goes to zero as $i \to \cU$ since $(a_i)_{i \in I} \in Z(\prod\nolimits^\cU \cM_i)$, hence $(a_i)_{i \in I} = (E_i(a_i))_{i \in I}$.
\end{proof}

\subsection{Consequences in Model Theory} \label{ss.mt}
Before proving Theorem \ref{thm:model}, we set up a model-theoretic framework suitable for tracially complete
\cstar-algebras.

We refer to \cite{modelc} for all the necessary background concerning continuous model theory of \cstar-algebras (see also \cite{clogic}
for a more general approach beyond the context of operator algebras). In \cite{modelc} unital \cstar-algebras are presented as
multi-sorted structures in the language $\cLc = \{ \lVert \cdot \rVert, + , \cdot, ^*, \{z \}_{z \in \mathbb{C}}, 0, 1 \}$
whose sorts, interpreted as the closed balls, represent the domains of quantification. This language is not handy for studying tracially complete \cstar-algebras, whose behavior is closer to that of tracial von Neumann algebras (see e.g. \cite{model_cstar:II}).

Anticipating the model
theoretic analysis of tracially complete \cstar-algebras which will be presented in \cite{tracialtransfer}, let $\cLt = \{ \lVert \cdot \rVert_2, + , \cdot, ^*, \{z \}_{z \in \mathbb{C}},0 ,1 \}$ be a language with a single sort and countably many domains $D_k$,
with two constant symbols $0$ and $1$, symbols for the algebraic operations $+,\cdot$, and  $^*$,  a symbol $z$ for each $z\in \mathbb{C}$ and a symbol for the tracial norm $\lVert \cdot \rVert_2$. 
The moduli of continuity assigned to the algebraic operations are chosen in the natural fashion. This is similar to the language
considered in \cite{model_cstar:II} for von Neumann algebras with the exception of the predicate $\text{tr}$, interpreted on tracial
von Neumann algebras as the trace.

A tracially complete \cstar-algebra $(\cM, X)$ is an $\cLt$-structure with the symbol $\lVert \cdot \rVert_2$ interpreted as $\lVert \cdot \rVert_{2,X}$, the operation symbols interpreted in the obvious way, and $D_k$ interpreted as the $k$-ball in the operator norm.

Given a sequence $((\cM_i, X_i) \mid i \in I)$ of tracially complete \cstar-algebras, it is possible to check that
the tracial ultraproduct $\prod^\cU \cM_i$ introduced in
\S\ref{ss.ultra} corresponds to the standard ultraproduct of metric structures (see \cite[Section 5]{clogic})
obtained when considering each $(\cM_i, X_i)$ as an $\cLt$-structure. In fact, the ultraproduct of the norms $\lVert \cdot \rVert_{2, X_i}$
can be verified to be precisely $\lVert \cdot \rVert_{2, \sum^\cU X_i}$, that is
\[
\lVert (a_i)_{i \in I} \rVert_{2, \sum^\cU X_i} = \lim_{i \to \cU} \lVert a_i \rVert_{2, X_i}, \text{ for every } (a_i)_{i \in I} \in \prod\nolimits^\cU \cM_i.
\]
This implies in particular that the unit ball of
$\prod^\cU \cM_i$ is $\lVert \cdot \rVert_{2, \sum^\cU X_i}$-complete (see \cite[Proposition 5.3]{clogic}), so it automatically follows that
$(\prod^\cU \cM_i, \sum^\cU X_i)$ is a tracially complete \cstar-algebra.
Through this section we stress the fact that we consider $\prod^\cU \cM_i$ as an
$\cLt$-structure with the 2-norm induced by $\sum^\cU X_i$ by saying that the ultraproduct of the sequence $((\cM_i, X_i) \mid i \in I)$
is the pair $(\prod^\cU \cM_i, \sum^\cU X_i)$.

In case the sequence $((\cM_i, X_i) \mid i \in I)$ is composed of $W^*$-bundles, with $X_i = \text{Prob}(K_i)$,
then the $W^*$-bundle
structure induced on $\prod^\cU \cM_i$ over $K= \sum^\cU K_i$ as in Proposition \ref{prop:ultra_w} makes the pair
$(\prod^\cU \cM_i, \text{Prob}(K))$ a tracially complete \cstar-algebra (see the discussion preceding Proposition \ref{prop:factor}). On the other hand, the ultraproduct of $((\cM_i, X_i) \mid i \in I)$
as $\cLt$-strcutres is $(\prod^\cU \cM_i, \sum^\cU X_i)$. These are two presentations of the same object, in fact $\sum^\cU X_i
= \text{Prob}(K)$. Indeed, given $(a_i)_{i \in I} \in \prod^\cU \cM_i$ we have
\begin{align} \label{eq:wultra}
\sup_{\tau \in \sum^\cU X_i} \lvert \tau((a_i)_{i \in I}) \rvert & = \lim_{i \to \cU} \sup_{\tau \in X_i} \lvert \tau(a_i) \rvert \\ \nonumber
&=   \lim_{i \to \cU} \sup_{\tau \in K_i} \lvert \tau(a_i) \rvert \\ \nonumber
&= \sup_{\tau \in K} \lvert \tau((a_i)_{i \in I}) \rvert \\\nonumber
&= \sup_{\tau \in \text{Prob}(K)} \lvert \tau((a_i)_{i \in I}) \rvert.
\end{align}
An application of the Hahn--Banach Theorem as in \cite[Lemma 4.4]{CETWW} then yields $\sum^\cU X_i
= \text{Prob}(K)$, since they are both convex and closed. In particular, in this case $\partial_e ( \sum^\cU X_i) = \sum^\cU \partial_e X_i$.

Given this model-theoretic setup, a direct application of \L o\'s Theorem gives
Theorem \ref{thm:Wbundle_trace} for ultrapowers over sets of indices different from $\bbN$.

\begin{proof}[Proof of Theorem \ref{thm:Wbundle_trace}]
Let $( \cM_i )_{i \in I}$ be a sequence of factorial $W^*$-bundles over $K_i$ and suppose there is $d \in \bbN$ such that $\text{dim}(K_i) \le d$
for every $i \in I$. Then $\cM = \prod^\cU \cM_i$ is a $W^*$-bundle over $K = \sum^\cU K_i$ and, as in the proof of Theorem
\ref{thm:Wbundle_trace_real} it is sufficient to show that for every $a \in \cM_{sa}$
\begin{equation} \label{eq:equality2}
\sup_{\lambda \in K} | \tau_\lambda(a) | = \sup_{\tau \in T(\cM)} | \tau(a) |.
\end{equation}

Consider the formula
\[
\phi(x,y) = \inf_{\substack{w_1, \dots, w_{10d}  \\ z_1, \dots, z_{10d}}} \Big \lVert x - y - 48 \sum_{i \le 10d} [w_i,z_i] \Big \rVert_2,
\]
where the $\inf$ ranges over the sort corresponding to the unit ball. Fix a contraction $a = (a_i)_{i \in I} \in \cM$.
Lemma \ref{lemma:fdlh_bundles} entails that $\phi^{(\cM_i, \text{Prob}(K_i))}(a_i, E_i(a_i)) = 0$ for every $i \in I$. Then, by \L o\'s's Theorem
(\cite[Theorem 5.4]{clogic}), we also have
\begin{equation} \label{eq:formula}
\phi^{(\cM, \text{Prob}(K))}(a, E^\cU(a)) = 0.
\end{equation}
In case $\sup_{\lambda \in K} | \tau_\lambda(a) | \le \delta$ for some $\delta \ge 0$ then $\lVert E^\cU(a) \rVert \le \delta$, hence,
arguing as in the proof of Theorem \ref{thm:Wbundle_trace_real}, it follows that 
\[
\sup_{\tau \in T(\cM)} | \tau(a) | \le \delta.
\]
The equality in \eqref{eq:equality2} follows since $\delta$ was chosen arbitrarily.
\end{proof}

The \cstar-algebras $A$ and $B$ considered in the statement of Theorem \ref{thm:model} are assumed to be elementarily
equivalent as $\cLc$-structures. In order to use the tools developed in the previous sections, we would like to be able to compare
their tracial completions, as defined in \cite{ozawa_dix}.
\begin{definition} \label{def:comp}
Given a unital \cstar-algebra $A$ with non-empty
trace space $T(A)$, its \emph{tracial completion} is the \cstar-algebra
\[
\overline{A}^{T(A)} = \frac{ \{ (a_n)_{n \in \bbN} \in \ell^\infty(A) \mid (a_n)_{n \in \bbN} \text{ is a $\lVert \cdot \rVert_{2, T(A)}$-Cauchy sequence} \}}{\{(a_n)_{n \in \bbN} \in \ell^\infty(A) \mid \lim_{n \to \infty} \| a_n \|_{2, T(A)} = 0\} }
\]
Every trace
$\tau \in T(A)$ canonically extends to a trace $\bar \tau$ on $\overline{A}^{T(A)}$, defined on each representing Cauchy sequence
$(a_n)_{n \in \bbN}$ as
\[
\bar \tau((a_n)_{n \in \bbN}) = \lim_{n \to \infty} \tau(a_n).
\]
We can thus identify $T(A)$ with a subset of the trace space of $\overline{A}^{T(A)}$, and it is immediate to check that
$(\overline{A}^{T(A)}, T(A))$ is a factorial tracially complete \cstar-algebra.
\end{definition}

Lemma \ref{lemma:eetracial} below is a straightforward generalization of \cite[Proposition 3.5.1]{modelc}, and it shows
that the first order $\cLc$-theory of a \cstar-algebra completely determines the $\cLt$-theory of its tracial completion for classes where the Cuntz--Pedersen
nullset is definable, in the sense of \cite[Chapter 3]{modelc}. Given a \cstar-algebra $A$, its \emph{Cuntz--Pedersen nullset} $A_0$,
introduced in \cite{cun_ped}, is the norm-closure of the linear span
of the set of self-adjoint commutators $[a, a^*]$. \cite[Theorem 2.9]{cun_ped} tightly relates the 2-norm $\lVert \cdot \rVert_{2,T(A)}$ of
an element with its distance from $A_0$, more precisely it shows that
\begin{equation} \label{eq:cp}
\lVert a \rVert_{2,T(A)}^2  = \sup_{\tau \in T(A)} \tau(a^*a) = d(a^*a, A_0), \ \forall a \in A.
\end{equation}

\cite[Theorem 3.5.5]{modelc} lists various classes where the Cuntz--Pedersen nullset is $\cLc$-definable. Examples are the set of exact $\mathcal{Z}$-stable
\cstar-algebras, and the collection of \cstar-algebras with strict comparison, to which the following lemma and Theorem \ref{thm:model}
apply.

\begin{lemma} \label{lemma:eetracial}
Let $A$ and $B$ be \cstar-algebras belonging to a class where the Cuntz-Pedersen nullset is definable.
If $A \equiv B$ as $\cLc$-structures, then $(\overline{A}^{T(A)} ,T(A)) \equiv (\overline{B}^{T(B)}, T(B))$ as $\cLt$-structures.
\end{lemma}
\begin{proof}
This fact is a consequence of \cite[Proposition 3.5.1]{modelc} in case both $T(A)$ and $T(B)$ are singletons.
More generally, given a \cstar-algebra $A$, definability of $A_0$
implies that the norm $\lVert \cdot \rVert_{2,T(A)}$ is also a definable predicate, by \cite[Theorem 3.2.2]{modelc} and \eqref{eq:cp}.

It follows that if $A$ is a \cstar-algebra belonging to a class where the Cuntz-Pedersen nullset is definable, then all interpretations of
$\mathcal{L}_2$-formulas on $A$ are  $\cLc$-definable predicates. This can be proved arguing by induction on the complexity of formulas,
with the atomic case being covered since $\lVert \cdot \rVert_{2,T(A)}$ is a definable predicate and the quantifier case since $A_1$ is dense in $(\overline{A}^{T(A)})_1$.
\end{proof}

The following lemma shows that isomorphisms between tracially
complete \cstar-algebras preserving the 2-norm also induce affine homeomorphisms between the sets of traces inducing the 2-norms. The argument, an application of the Hahn--Banach Theorem, is due
to the authors of \cite{tracially_complete}.
\begin{lemma}[{\cite{tracially_complete}}] \label{lemma:iso_simplex}
Let $(\cM, X)$ and $(\cN, Y)$ be two tracially complete \cstar-algebra and let $\theta\colon \cM \to \cN$ be an isomorphism such that
$\lVert a \rVert_{2, X} = \lVert \theta(a) \rVert_{2, Y}$ for every $a \in \cM$. Then $\theta^*(Y) = X$.
\end{lemma}
\begin{proof}
Suppose that $\theta^*(Y) \not = X$ and that there is $\tau \in X \setminus \theta^*(Y)$ (the case where $\tau \in \theta^*(Y) \setminus
X$ is analogous). As $\theta^*(Y)$ is closed and convex, by the
Hahn--Banach Theorem there exists an affine, positive, continuous function $f$ on $T(\cM)$
and $\alpha \in \mathbb{R}$ such that $f(\tau) > \alpha$ and $f(\sigma) < \alpha$ for every $\sigma \in \theta^*(Y)$. By \cite[Lemma 6.2]{kirchror}
we can approximate $f$ arbitrarily well with evaluations on positive elements of $\cM$, hence there exists $a \in \cM_+$ such that
\[
\sup_{\sigma \in \theta^*(Y)} \sigma(a) < \tau(a).
\]
This gives $\lVert \theta(a^{1/2}) \rVert_{2, Y} < \lVert a^{1/2} \rVert_{2, X}$, which contradicts the assumption on $\theta$.
\end{proof}

In the next proof we repeatedly use the fact, a consequence of a standard density argument, that the tracial ultraproduct $A^\cU$
of a unital \cstar-algebra $A$ is equal to the ultrapower of its tracial completion $(\overline{A}^{T(A)})^\cU_{T(A)}$.

\begin{theorem}
Let $A,B$ be two unital \cstar-algebras belonging to any class where the Cuntz--Pedersen nullset is definable.
Suppose that $A \equiv B$ (as $\cLc$-structures). Then $T(A)$ is a Bauer simplex if and only if 
$T(B)$ is. If moreover both $T(A)$ and $T(B)$ are Bauer simplices and second countable,
then $\partial_e T(A)$ and $\partial_e T(B)$ have the same covering dimension.
\end{theorem}
\begin{proof}
Fix $A$ and $B$ as in the statement. By Lemma \ref{lemma:eetracial} their tracial completions
are elementarily equivalent as $\cLt$-structures, thus by \cite[Theorem 5.7]{clogic} there are an index set $I$ and an ultrafilter $\cU$ over $I$ such that $(A^\cU, T(A)^\cU) \cong (B^\cU, T(B)^\cU)$
with an isomorphism which is $\lVert \cdot \rVert_{2, T(A)^\cU}-\lVert \cdot \rVert_{2, T(B)^\cU}$-isometric. Suppose that $T(A)$
is a Bauer simplex. By Theorem \ref{thm:bauer} and Proposition \ref{prop:ultra_w} the pair $(A^\cU, T(A)^\cU)$ can be
endowed with a $W^*$-bundle structure over $(\partial_e T(A))^\cU = \partial_e (T(A)^\cU)$ and thus $T(A)^\cU$ is a Bauer simplex
(see also the discussion preceding the computation in \eqref{eq:wultra}).
This in turn implies, by Lemma \ref{lemma:iso_simplex}, that $T(B)^\cU$ is a Bauer simplex.

Fix $\tau \in \overline{\partial_e T(B)}$. We prove that the canonical extension of $\tau$ to $B^\cU$ (which we still denote $\tau$) belongs to the boundary
of $T(B)^\cU$. This is sufficient to conclude that $\tau \in \partial_e T(B)$, since any non-trivial convex decomposition of $\tau$ in $T(B)$ induces a non-trivial
convex decomposition of $\tau$ as an element of $T(B)^\cU$. The simplex $T(B)^\cU$ is Bauer, hence it is enough to prove that $\tau$ is in the closure of
$\partial_e (T(B)^\cU)$. To this end, pick $\e > 0$ and $a_1, \dots, a_m \in B^\cU$.
Fix a representing sequence
$(a_{j,i})_{i \in I}$ of $a_j$ for all $j \le m$. For every $i \in I$, let $\sigma_i \in \partial_e T(B)$ be such that
\begin{equation} \label{eq:aij}
\max_{j \le m} \{ \lvert \sigma_i(a_{j,i}) - \tau(a_{j,i}) \rvert \} < \e.
\end{equation}
Notice that the limit trace $\sigma =(\sigma_i)_{i \in I}$ is extremal in $T(B^\cU)$. This is the case since by uniqueness of the GNS representation,
the identity map on $B^\cU$ induces a surjective isomorphism
\[
\pi_{\sigma}(B^\cU)'' \to \prod\nolimits^\cU \pi_{\sigma_i}(B)'',
\]
between the von Neumann algebra $\pi_{\sigma}(B^\cU)''$ generated by the GNS-representation corresponding to $\sigma$,
and the tracial ultraproduct of those corresponding to $\sigma_i$.
Each $\sigma_i$ is extremal, hence every $\pi_{\sigma_i}(B)''$ is factor, and so is their ultraproduct.
We conclude that $\pi_{\sigma}(B^\cU)''$ is a factor and therefore that $\sigma$ is an extremal point in $T(B^\cU)$ approximating $\tau$.

We rely on Theorem \ref{thm:center} to prove the second part of the statement. More in detail, the completion $\overline{A}^{T(A)}$ is a
factorial $W^*$-bundle over $\partial_e T(A)$ by Theorem \ref{thm:bauer}, and its center is equal to $C(\partial_e T(A))$ by factoriality.
The same holds for $B$. The relation $A^\cU \cong B^\cU$ implies $Z(A^\cU) \cong Z(B^\cU)$, hence Theorem \ref{thm:center} gives $C(\partial_e T(A))^\cU \cong C(\partial_e T(B))^\cU$. This in turn gives $C(\partial_e T(A)) \equiv C(\partial_e T(B))$ as
$\cLc$-structures, since for abelian \cstar-algebras the tracial ultrapower is equal to the \cstar-norm ultrapower.
The covering dimension of $X$ is equal to the decomposition rank of $C(X)$ (\cite[Proposition 3.3]{kirch_win}; this is the only
step where the fact that $T(A)$ and $T(B)$ are second countable is used), and the latter is definable
by uniform families of formulas (\cite[Theorem 5.7.3]{modelc}), therefore the conclusion follows.
\end{proof}

\bibliographystyle{amsalpha}
\bibliography{Bibliography}

\newcommand{\etalchar}[1]{$^{#1}$}
\providecommand{\bysame}{\leavevmode\hbox to3em{\hrulefill}\thinspace}
\providecommand{\MR}{\relax\ifhmode\unskip\space\fi MR }
\providecommand{\MRhref}[2]{%
  \href{http://www.ams.org/mathscinet-getitem?mr=#1}{#2}
}
\providecommand{\href}[2]{#2}
\begin{thebibliography}{BYBHU08}

\bibitem[Alf71]{alfsen}
E.M. Alfsen, \emph{Compact convex sets and boundary integrals}, Ergebnisse der
  Mathematik und Ihrer Grengzgbiete, vol.~57, Springer, 1971.

\bibitem[APRT23]{APRT}
R.~Antoine, F.~Perera, L.~Robert, and H.~Thiel, \emph{Traces on ultrapowers of
  \cstar-algebras}, preprint arXiv:2303.01929 (2023).

\bibitem[BBS{\etalchar{+}}19]{BBSTWW}
J.~Bosa, N.~P. Brown, Y.~Sato, A.~Tikuisis, S.~White, and W.~Winter,
  \emph{Covering dimension of {$\rm C^*$}-algebras and 2-coloured
  classification}, Mem. Amer. Math. Soc. \textbf{257} (2019), no.~1233, vii+97.
  \MR{3908669}

\bibitem[BF15]{BiFa}
T.~Bice and I.~Farah, \emph{Traces, ultrapowers and the {P}edersen-{P}etersen
  {$C^\ast$}-algebras}, Houston J. Math. \textbf{41} (2015), no.~4, 1175--1190.
  \MR{3455354}

\bibitem[BK04]{blan_kirch}
E.~Blanchard and E.~Kirchberg, \emph{Global {G}limm halving for
  {$C^*$}-bundles}, J. Operator Theory \textbf{52} (2004), no.~2, 385--420.
  \MR{2120237}

\bibitem[Bla06]{blackadar}
B.~Blackadar, \emph{Operator algebras}, Encyclopaedia of Mathematical Sciences,
  vol. 122, Springer-Verlag, Berlin, 2006, Theory of $C^*$-algebras and von
  Neumann algebras, Operator Algebras and Non-commutative Geometry, III.
  \MR{2188261}

\bibitem[BYBHU08]{clogic}
I.~Ben~Yaacov, A.~Berenstein, C.~W. Henson, and A.~Usvyatsov, \emph{Model
  theory for metric structures}, Model theory with applications to algebra and
  analysis. {V}ol. 2, London Math. Soc. Lecture Note Ser., vol. 350, Cambridge
  Univ. Press, Cambridge, 2008, pp.~315--427. \MR{2436146}

\bibitem[CCE{\etalchar{+}}]{tracially_complete}
J.~Carri\'on, J.~Castillejos, S.~Evington, J.~Gabe, C.~Schafhauser,
  A.~Tikuisis, and S.~White, \emph{Tracially complete \cstar-algebras},
  Manuscript in preparation.

\bibitem[CET{\etalchar{+}}21]{CETWW}
J.~Castillejos, S.~Evington, A.~Tikuisis, S.~White, and W.~Winter,
  \emph{Nuclear dimension of simple {$\rm C^*$}-algebras}, Invent. Math.
  \textbf{224} (2021), no.~1, 245--290. \MR{4228503}

\bibitem[CETW21]{CETW}
J.~Castillejos, S.~Evington, A.~Tikuisis, and S.~White, \emph{Classifying maps
  into uniform tracial sequence algebras}, M\"{u}nster J. Math. \textbf{14}
  (2021), no.~2, 265--281. \MR{4359832}

\bibitem[CP79]{cun_ped}
J.~Cuntz and G.~K. Pedersen, \emph{Equivalence and traces on {$ C^\ast
  $}-algebras}, J. Funct. Anal. \textbf{33} (1979), no.~2, 135--164.
  \MR{546503}

\bibitem[EP16]{ev_pen}
S.~Evington and U.~Pennig, \emph{Locally trivial {${ W}^*$}-bundles}, Internat.
  J. Math. \textbf{27} (2016), no.~11, 1650088, 25. \MR{3570373}

\bibitem[Evi18]{sam_thesis}
S.~Evington, \emph{{$W^*$}-bundles}, Ph.D. thesis, Univeristy of Glasgow, 2018.

\bibitem[FdlH80]{fackdlh}
T.~Fack and P.~de~la Harpe, \emph{Sommes de commutateurs dans les alg\`ebres de
  von {N}eumann finies continues}, Ann. Inst. Fourier (Grenoble) \textbf{30}
  (1980), no.~3, 49--73. \MR{597017}

\bibitem[FHH{\etalchar{+}}]{tracialtransfer}
I.~Farah, B.~Hart, I.~Hirshberg, C.~Schafhauser, A.~Tikuisis, and A.~Vaccaro,
  \emph{A tracial transfer property}, Manuscript in preparation.

\bibitem[FHL{\etalchar{+}}21]{modelc}
I.~Farah, B.~Hart, M.~Lupini, L.~Robert, A.~Tikuisis, A.~Vignati, and
  W.~Winter, \emph{Model theory of {$C^*$}-algebras}, Mem. Amer. Math. Soc.
  \textbf{271} (2021), no.~1324, viii+127. \MR{4279915}

\bibitem[FHS13]{model_cstar:I}
I.~Farah, B.~Hart, and D.~Sherman, \emph{Model theory of operator algebras {I}:
  stability}, Bull. Lond. Math. Soc. \textbf{45} (2013), no.~4, 825--838.
  \MR{3081550}

\bibitem[FHS14]{model_cstar:II}
\bysame, \emph{Model theory of operator algebras {II}: model theory}, Israel J.
  Math. \textbf{201} (2014), no.~1, 477--505. \MR{3265292}

\bibitem[Goo86]{goodearl}
K.~R. Goodearl, \emph{Partially ordered abelian groups with interpolation},
  Mathematical Surveys and Monographs, vol.~20, American Mathematical Society,
  Providence, RI, 1986. \MR{845783}

\bibitem[KR14]{kirchror}
E.~Kirchberg and M.~R{\o}rdam, \emph{Central sequence {$C^*$}-algebras and
  tensorial absorption of the {J}iang-{S}u algebra}, J. Reine Angew. Math.
  \textbf{695} (2014), 175--214. \MR{3276157}

\bibitem[KW04]{kirch_win}
E.~Kirchberg and W.~Winter, \emph{Covering dimension and quasidiagonality},
  Internat. J. Math. \textbf{15} (2004), no.~1, 63--85. \MR{2039212}

\bibitem[Lia16]{liao:Z}
H.-C. Liao, \emph{A {R}okhlin type theorem for simple {$C^*$}-algebras of
  finite nuclear dimension}, J. Funct. Anal. \textbf{270} (2016), no.~10,
  3675--3708. \MR{3478870}

\bibitem[Lia17]{liao:Zm}
\bysame, \emph{Rokhlin dimension of {$\Bbb{Z}^m$}-actions on simple
  {$C^*$}-algebras}, Internat. J. Math. \textbf{28} (2017), no.~7, 1750050, 22.
  \MR{3667895}

\bibitem[MS12]{matuisato}
H.~Matui and Y.~Sato, \emph{Strict comparison and {$\mathcal{Z}$}-absorption of
  nuclear {$C^*$}-algebras}, Acta Math. \textbf{209} (2012), no.~1, 179--196.
  \MR{2979512}

\bibitem[Mur90]{murphy}
G.~J. Murphy, \emph{{$C^*$}-algebras and operator theory}, Academic Press,
  Inc., Boston, MA, 1990. \MR{1074574}

\bibitem[Oza13]{ozawa_dix}
N.~Ozawa, \emph{Dixmier approximation and symmetric amenability for {$
  C^*$}-algebras}, J. Math. Sci. Univ. Tokyo \textbf{20} (2013), no.~3,
  349--374. \MR{3156986}

\bibitem[PP70]{PP}
G.~K. Pedersen and N.~H. Petersen, \emph{Ideals in a {$C\sp{\ast} $}-algebra},
  Math. Scand. \textbf{27} (1970), 193--204 (1971). \MR{308797}

\bibitem[Roy75]{roy}
A.~K. Roy, \emph{Closures of faces of compact convex sets}, Ann. Inst. Fourier
  (Grenoble) \textbf{25} (1975), no.~2, 221--234. \MR{390714}

\bibitem[RW98]{RaeWill}
I.~Raeburn and D.~P. Williams, \emph{Morita equivalence and continuous-trace
  {$C^*$}-algebras}, Mathematical Surveys and Monographs, vol.~60, American
  Mathematical Society, Providence, RI, 1998. \MR{1634408}

\bibitem[Sat12]{sato}
Y.~Sato, \emph{Trace spaces of simple nuclear \cstar-algebras with
  finite-dimensional extreme bound- ary}, preprint arXiv:1209.3000 (2012).

\bibitem[TWW15]{TWW}
A.~S. Toms, S.~White, and W.~Winter, \emph{{$\mathcal{Z}$}-stability and
  finite-dimensional tracial boundaries}, Int. Math. Res. Not. (2015), no.~10,
  2702--2727. \MR{3352253}

\bibitem[Wou21]{wouters}
L.~Wouters, \emph{Equivariant $\mathcal{Z}$-stability for single automorphisms
  on simple \cstar-algebras with tractable trace simplices}, preprint
  arXiv:2105.04469 (2021).

\end{thebibliography}
\end{document}